\documentclass{amsart}
\usepackage{amssymb,amscd}
\usepackage{times}
\usepackage{enumerate}
\usepackage[11pt,curve,matrix,arrow,frame,graph]{xy}
\newcounter{intro}

\newtheorem{theo}[intro]{Theorem}

\newtheorem{thm}{Theorem}[section]

\newtheorem{prop}[thm]{Proposition}
\newtheorem{cor}[thm]{Corollary}

\theoremstyle{remark}
\newtheorem{rem}[thm]{Remark}
\newtheorem{rems}[thm]{Remarks}

\newtheorem{defi}[thm]{Definition}
\newtheorem*{merci}{Acknowledgements}

\numberwithin{equation}{section}   

\newcounter{counteroman}

\newcommand{\cref}[1]{Corollary~\ref{#1}}

\newcommand{\R}{\mathbb{R}}

\newcommand{\bg}{\mathbf{g}}
\newcommand{\bB}{\mathbb{B}}

\newcommand{\bS}{\mathbb{S}}

\newcommand{\mcC}{\mathcal{C}}

\newcommand{\mcM}{\mathcal{M}}

\newcommand{\bpm}{\begin{pmatrix}}
\newcommand{\epm}{\end{pmatrix}}

\let\eps=\varepsilon

\DeclareMathOperator{\cour}{\mathcal{R}}

\DeclareMathOperator{\Rm}{Rm}\DeclareMathOperator{\We}{W}
\DeclareMathOperator{\ricci}{Ricci}
\DeclareMathOperator{\scal}{Scal}

\DeclareMathOperator{\vol}{vol}

\def\cro#1#2{\mathrel{\langle {#1},{#2}\rangle}}

\def\and{\ \mathrm{ and}\,Ê}

\begin{document}

\title[Some old and new results about rigidity of critical metric]
{Some old and new results about rigidity of critical metric}

\author{Gilles Carron}
\address{Laboratoire de Math\'ematiques Jean Leray (UMR 6629), Universit\'e de Nantes, 
2, rue de la Houssini\`ere, B.P.~92208, 44322 Nantes Cedex~3, France}
\email{Gilles.Carron@math.univ-nantes.fr}
\date{\today}


\begin{abstract}
We present a new proof of a recent $\epsilon$ regularity of G. Tian and J.Viaclovsky. Moreover, our idea also also works
with a kind of $L^p, p<\dim M/2$ assumptions on the curvature.
\end{abstract}

\maketitle
\section{Introduction}
In this paper, we obtain some new $\epsilon$-regularity and rigidity results for critical metrics and our arguments will also give new proof of classical $\epsilon$-regularity results. 

The class of critical metric has been introduced and study by G. Tian and J. Viaclovsky (\cite{TV_inven}) :
A Riemannian metric $g$ is said to be \textit{critical} if its Ricci curvature tensor satisfies a Bochner type formula :
$$\nabla^*\nabla\ricci_g+\cour(\ricci_g)=0$$
where $\nabla$ is the Levi-Civita connection and $\nabla^*$ is its differential adjoint and $\cour$ is a linear action of the Riemann curvature tensor on the space of symmetric $2$ tensor, in particular there is a constant $\Upsilon$ (that only depends on this action) such that :
\begin{equation}\label{constant}\forall h\in \odot^2 T^*_xM, \,\,\, |\cour(h)|\le \Upsilon |\Rm|\, |h|.\end{equation}
Examples of critical metric are Einstein metric, K\"ahler metric with constant scalar curvature, locally conformally flat metric with constant scalar curvature and in dimension $4$, Bach flat metric with constant scalar curvature.
Our main new result is the following $\epsilon$-rigidity result :
\begin{theo} There is a constant $\epsilon>0$ that only depends on the dimension $n$ and of the constant $\Upsilon$ appearing in the estimate
(\ref{constant}) such that if $(M^n ,g)$ be a complete Riemannian manifold whose metric is critical and such that its Riemann curvature tensor satisfies for some fixed point $o\in M$ :
$$|\Rm(y)|\le \frac{\epsilon^2}{d(o,y)^2}\,\,\,,$$
then the metric $g$ is flat : $\Rm=0$.
\end{theo}
Our result generalizes a recent result of V. Minerbe (\cite{Vincent_GAFA}) who proved a similar result for Ricci flat metric with controlled volume growth :
\begin{thm}\label{VM}Assume that $(M^n,g)$ is complete Ricci flat Riemannian manifold $$\ricci_g=0$$ such that for some fixed point
 $o\in M$, some $\nu>1$ and some positive constant $C>0$ :
$$\forall R>r>0,\,\,\, \frac{\vol B(o,R)}{\vol B(o,r)}\ge C\left(\frac{R}{r}\right)^\nu$$
then there is a constant $\epsilon>0$  that only depends on $n,\nu,C$ such that if
$$|\Rm(y)|\le \frac{\epsilon^2}{d(o,y)^2}\,\,\,,$$
then the metric $g$ is flat : $\Rm=0$.
\end{thm}
The first step in the proof of theorem \ref{VM} was to establish a $L^1$ Hardy inequality  :
$$\forall f\in C_0^\infty(M), \mu(n,\nu,C)\int_M \frac{|f(x)|}{d(x,o)}d\vol_g(x)\le \int_M |df(x)| d\vol_g(x).$$
And the final step was to use the Bochner type equation 
$$\nabla^*\nabla\Rm+\cour(\Rm)=0$$ satisfied by the Riemann curvature tensor of a Ricci flat metric.

There are many other $\epsilon$-rigidity results that relies on a priori functional inequality (such as a Sobolev inequality or as the above Hardy inequality) and a integral bounds on the curvature (cf. for instance \cite{Berard}, \cite{Singer}, \cite{shen1},\cite{shen2}, \cite{Ni},\cite{IS},\cite{PRS}, \cite[Theorem 7.1]{TV_inven},\cite{XuZhao},\cite{Kim} ). Such a result
has been shown recently for critical metric by G.Tian and J.Viaclovsky in dimension $4$ and by X-X. Chen and B.Weber in higher dimension (\cite{TV_inven}, \cite{ChenWeber} :
\begin{thm}\label{regula} There are positive constants $\epsilon>0$ and $C>0$ that depend only on the dimension $n$ and of the constant $\Upsilon$ appearing in the estimate
(\ref{constant}) such that when $(M^n ,g)$ be a complete Riemannian manifold whose metric is critical and such that for some $x\in M$ and $r>0$,  the geodesic ball $B(x,r)$ satisfies the Sobolev inequality :
$$\forall f\in C_0^\infty(B(x,r))\,\, ,\,\, \left(\int_{B(x,r)} |f(y)|^{\frac{2n}{n-2}}d\vol(y)\right)^{1-\frac2n}\le A\int_{B(x,r)}|df(y)|^2  d\vol(y)$$
and the following bound for the curvature tensor :
$$A^{\frac n2}\int_{B(x,r)}|\Rm(y)|^{\frac n2}d\vol_g(y)<\epsilon\, $$ then
$$\sup_{B(x,\frac12 r)} |\Rm|\le A\, \frac{C}{ r^2}\, \left(\int_{B(x,r)} |\Rm|^{\frac n2}(y)d\vol_g(y)\right)^{\frac 2n}.$$
\end{thm}
Such a result implies the following $\epsilon$-rigidity result :
\begin{cor}Let $(M^n ,g)$ be a complete Riemannian manifold whose metric is critical . Assume that$(M^n,g)$ satisfies the Sobolev inequality :
$$\forall f\in C_0^\infty(M)\,\, ,\,\, \left(\int_{M} |f(y)|^{\frac{2n}{n-2}}d\vol(y)\right)^{1-\frac2n}\le A\int_{M}|df(y)|^2  d\vol(y).$$
If the curvature tensor satisfies
$$A^{\frac n2}\int_{M}|\Rm(y)|^{\frac n2}d\vol_g(y)<\epsilon\, $$
then $(M^n,g)$ is isometric to the Euclidean space $\R^n$.
\end{cor}

In another paper \cite{TV_cmh}, G. Tian and J. Viaclovsky were able to replace the hypothesis on the Sobolev inequality by a uniform lower bound on the volume growth of geodesic balls :
$$\forall y\in B(x,r),\,\, \forall s\in (0,r)\,\,:\,\, \vol B(y,s)\ge v s^n$$
It is known that the Sobolev inequality implies such a uniform lower bound ((\cite{Aku} or \cite{Carron_smf}). The proof of this improvement used as a preliminary result the above $\epsilon$ regularity result (theorem \ref{regula}) and hence it relied on the intricate de Georgi-Nash-Moser iteration scheme developed in \cite{TV_inven} or \cite{ChenWeber}. Our idea leads to a direct proof of this improvement that do not used this iteration scheme and moreover we are able to give some $L^p$ $\epsilon$ regularity/rigidity result, for instance we'll obtain the following :
\begin{theo}There is a constant $\epsilon>0$ that only depends on $n,p$ and of  constant $\Upsilon$ appearing in the estimate
(\ref{constant}) such that when $(M^n ,g)$ be a complete Riemannian manifold whose metric is critical and such that 
any geodesic ball $B\subset M$ (with radius $r(B)$) satisfies \footnote{$\omega_n$ is the volume of the unit Euclidean ball.}:
$$\frac{r^{2p}}{\vol B}\int_B |\Rm(y)|^{p}d\vol_g(y)<\epsilon\, $$
then the metric $g$ is flat : $\Rm=0$.\end{theo}

Our argument also leads to a new and direct proof of the following result of M. Anderson :
 \begin{thm}There is a positive constant $\epsilon_n>0$ such that if $(M^n,g)$ is a complete Ricci flat manifold
satisfying :
$$\lim_{r\to\infty} \frac{\vol B(x,r)}{r^n}\ge \omega_n(1-\epsilon_n)$$
then $(M^n,g)$  is isometric to the Euclidean space $\R^n$.\end{thm}
This result was used by Anderson to prove a $\epsilon$-regularity result based on volume growth for metric with bounded Ricci curvature ; for Einstein metric, this result implies some uniform bound on the Riemann curvature tensor. In fact we obtain a new proof and a new formulation of this estimate :
\begin{theo}There are constant $\epsilon(n)>0$ and $C(n)$ such that if
$(M^n,g)$ is a complete Ricci flat manifold and $x\in M$ and $r>0$ are such that 
$$\vol B(x,r)\ge \omega_n(1-\epsilon_n)r^n$$ then
$$\sup_{B(x,r/2)}|\Rm|\le \frac{C(n)}{r^{2}}\sup_{y\in B(x,\frac34 r)}\left( \frac{\omega_n r^n-\vol B(y,r)}{r^n}\right)^{\frac14}.$$
\end{theo}

Our idea is quite versatile and can be used to obtain other rigidity and regularity results. In a future work, we intend to consider applications of these ideas to the question of convergence of Einstein/critical metric in dimension $n>4$ in the spirit of results of J. Cheeger, T. Colding, G. Tian \cite{CCT_GAFA} or of G. Tian and J. Viaclovsky \cite{TV_adv}. What nowadays is missing is an answer to a question of M. Anderson (cf. \cite[Rem 2, p. 475]{And_JAMS} and G. Tian \cite{Tian_ICM} about the geometry of Einstein/critical Riemannian manifold with maximal volume growth and whose curvature satisfies some bound on :
$$\sup_r \left(r^{4-n}\int_{B(x,r)}|\Rm|^2\right).$$
\begin{merci}
I would like to thank E. Aubry, P. Castillon, R. Mazzeo, V. Minerbe  for helpful discussions.
I was partially supported by  by the grant GeomEinstein 06-BLAN-0154.
\end{merci}
\section{Some Definitions and useful tools}

\subsection{Regular metric}
\begin{defi}
We say that  a Riemannian manifold $(M^n,g)$ is $(\Lambda,k)$ regular or that the Riemannian metric $g$  satisfies $(\Lambda,k)$ regularity estimates
 if for any $x\in M$ and any $r>0$ and $\eps\in(0,1)$ such that
$$\sup_{ÊB(x,\eps r)} |\Rm|\le \frac{1}{r^2}$$ then for all $j=1, \hdots , k$
\begin{equation}\label{eq:regularestimate}
\sup_{ÊB\left(x,\frac{1}{2}\eps r\right)} |\nabla^j\Rm|\le \frac{\Lambda}{(\eps r)^j r^2}.
\end{equation}
\end{defi}

\begin{rems}
\begin{enumerate}[a)]
\item  The choice of half the radius in the estimate (\ref{eq:regularestimate}) is arbitrary, indeed it is easy to show that the $(\Lambda,k)$ regularity estimate implies the following : if for some $x\in M$, $r>0$ and $\eps\in (0,1)$ we have
$$\sup_{ÊB(x,\eps r)} |\Rm|\le \frac{1}{r^2}$$ then for all $\delta\in (0,1)$ and all $j=1, \hdots , k$, we have
$$
\sup_{ÊB\left(x,\delta\, \eps r \right)} |\nabla^j\Rm|\le \frac{\Lambda}{\eps^j ((1-\delta)r)^{2+j}}.
$$
\item This  condition of regularity is clearly invariant by scaling : if a metric $g$ satisfies $(\Lambda,k)$ regularity estimates then
for any positive constant $h$, the metric $h^2 g$ satisfies $(\Lambda,k)$ regularity estimates. 

\item Hence, a metric $g$ satisfies $(\Lambda,k)$ regularity estimates if and only if for every positive constant $h$ the metric
$g_h=h^2g$ satisfies the following estimates :
for any $x\in M$ and any  $\eps\in(0,1)$ such that
$$\sup_{ÊB_{g_h}(x,1)} |\Rm_{g_h}|\le \eps^2$$ then for all $j=1, \hdots , k$
 $$\sup_{ÊB_{g_h}\left(x,\frac{1}{2}\right)} |\nabla^j\RmÊ_{g_h}|\le  \Lambda\, \eps ^2.$$
\end{enumerate}
\end{rems}

Sometimes, we will used a weaker assumption on the metric :
\begin{defi}
We say that a Riemannian manifold $(M^n,g)$ is weakly $(\Lambda,k)$ regular if
 if for any $x\in M$ and any $r>0$ such that
$$\sup_{ÊB(x, r)} |\Rm|\le \frac{1}{r^2}$$ then for all $j=1, \hdots , k$

$$ \sup_{ÊB\left(x,\frac{ r}{2}\right)} |\nabla^j\Rm|\le \frac{\Lambda}{ r^{j +2}}.$$
\end{defi}

\subsection{Examples of Regular metric}
\subsubsection{Einstein metric and metric with harmonic curvature} When $(M^n,g)$ is Einstein $$\ricci_g=(n-1)\tau g$$
then the curvature satisfies an elliptic equation 
\begin{equation}\label{eq:Boch_curvature}
\nabla^*\nabla \Rm+\cour(\Rm)=0\end{equation}
where $\cour$ is a certain action of the curvature operator on the space of curvature tensors.
Indeed the Bianchi identies implies that
$$d^\nabla\Rm=0$$ and the fact that the Ricci curvature is zero implies that the curvature tensor (viewed as a $2$-forms valued in symmetric tensors)
is coclosed :
$$(d^\nabla)^*\Rm=0\,\,,$$ hence the above equation (\ref{eq:Boch_curvature}) is a consequence of a Bochner formula (cf \cite[Proposition 4.2]{Bourguignon_inv})
$$(d^\nabla)^*d^\nabla+d^\nabla(d^\nabla)^*=\nabla^*\nabla +\cour.$$
So that any harmonic Riemann tensor  :
$$(d^\nabla)^*\Rm=0$$ satisfies the  Bochner formula  (\ref{eq:Boch_curvature}). This implies the following :
\begin{prop}
If $(M^n,g)$ is a Riemannian manifold with harmonic curvature :
$$\left(d^\nabla\right)^*\Rm=0$$ then $(M^n,g)$ is $(\Lambda,k)$ regular for a constant $\Lambda$ that only depends on $n$ and $k$.
\end{prop}

\begin{proof}This regularity result can be proved with some rather classical elliptic regularity estimates, along the line of the proof of regularity of critical metric (see the proof of proposition \ref{prop:regularcritical}). But we can also use less elaborate tools using only
the maximum principle (following for instance the argumentation of W. Shi, \cite[section 7]{Shi_JDG1}).

Indeed assume that $g$ is a complete Riemannian metric with harmonic curvature. If we assume that on a geodesic ball $B(x,1)\subset M$, and for some $\eps\in (0,1)$, we have the following uniform bound on the curvature :
$$\sup_{ÊB(x, 1)} |\Rm|\le \eps^2$$
Then the exponential map is a local diffeomorphism form the unit Euclidean ball $\bB(0,1)\subset (T_xM,g_x)$ to $B(x,1)$ :
$$\exp_x\,:\, \bB(0,1)\rightarrow B(x,1)$$
Then metric $\bg=\exp_x^*g$ has also a harmonic curvature tensor and its curvature tensor is bounded by $\eps^2$.
We will proved the regularity estimate in the ball $\bB(0,1)$ endowed with the metric $\bg=\exp_x^*g$. Hence we work now on the Riemannian manifold $(\bB(0,1), \bg)$
The Bochner's formulae imply that 
\footnote{$C(n)$ will be a constant that only depends on $n$ and that can vary from one estimate to another.}:

$$\Delta |\Rm|^2\le C(n)\epsilon^2|\Rm|^2-2|\nabla \Rm|^2$$
$$\and \,\,\,\, \Delta |\nabla\Rm|^2\le C(n)\epsilon^2|\nabla\Rm|^2-2|\nabla^2 \Rm|^2.$$

We define $v=(33\epsilon^4+ |\Rm|^2) |\nabla\Rm|^2$ and consider the function
$\varphi=2u-u^2$ where $$u(y)=\begin{cases}
1& \, \mathrm{if}\,Ê |y|\le 1/2\\
\left(3-4|y|^2\right)^2 &\, \mathrm{if}\,Ê 1/2\le |y|\le 3/4\\
0&\, \mathrm{if}\,  3/4\le |y|\end{cases}$$
then we have
$$|\Delta\varphi|\le C(n)\and |d\varphi|^2\le \varphi \,\,\, .$$
Hence at a point where the function $\varphi v$ is maximal we have
$$vd\varphi+\varphi dv=0\and 0\le \Delta (\varphi v)$$
Hence at such a point :
\begin{equation*}\begin{split}
0&\le v\Delta \varphi-2\cro{d\varphi}{dv}+\varphi\Delta v\\
&\le v\Delta \varphi+2\frac{|d\varphi|^2}{\varphi} v+\varphi\Delta v\\
&\le C(n) v+\varphi\Delta v\,\,\,\, .
\end{split}\end{equation*}

A quick computation shows that
\begin{equation*}\begin{split}
\Delta v&\le (C(n)\epsilon^2   |\Rm|^2-2 |\nabla\Rm|^2) |\nabla\Rm|^2+(33\epsilon^4+ |\Rm|^2) (C(n)\epsilon^2 |\nabla\Rm|^2-2|\nabla^2\Rm|^2)\\
&\,\,\,\,\,\,\,\,\,\,\,\,\,\,\,\,\,\,\,\,\,\,\,\,\,\,\,\,\,\,\,\,\,\,\,\,\,\,\,\,\,\,\,\,\,\,\,\,\,\,\,\,\,\,\,\,\,\,\,\,\,\,\,\,\,\,\,\,-2\cro{d |\Rm|^2}{d |\nabla\Rm|^2}\\
&\le C(n)\epsilon^2v-2|\nabla\Rm|^4-2(33\epsilon^4+ |\Rm|^2) |\nabla^2\Rm|^2+8 |\Rm| |\nabla\Rm|^2 |\nabla^2\Rm|\\
&\le C(n)\epsilon^2 v-|\nabla\Rm|^4-2(33\epsilon^4+ |\Rm|^2) |\nabla^2\Rm|^2+16|\Rm|^2 |\nabla^2\Rm|^2\\
&\le C(n)\epsilon^2 v-|\nabla\Rm|^4-2(33\epsilon^4+ |\Rm|^2) |\nabla^2\Rm|^2+16\eps^4 |\nabla^2\Rm|^2\\
&\le C(n)\epsilon^2 v-|\nabla\Rm|^4\end{split}\end{equation*}
Hence at a point where the function $\varphi v$ is maximal, we have
$$0\le C(n)v+ C(n)\epsilon^2\varphi v-\varphi\ |\nabla\Rm|^4 \,\, ,$$
so that we have at such a point :
$$\varphi^2\frac{v^2}{(34\eps^4)^2}\le \varphi^2\ |\nabla\Rm|^4\le C(n) \varphi v\, \, .$$
This estimate implies the following
$$ \sup_{\bB(0,1)}\varphi v\le C(n) \eps^8,$$
and with the definition of $v= (33\epsilon^4+ |\Rm|^2) |\nabla\Rm|^2$ , we get :
$$\sup_{\bB(0,\frac12)}|\nabla\Rm|^2\le  C(n) \eps^4.$$

The estimate on the higher order covariant derivative of the Riemann tensor $|\nabla^j\Rm|$ can be obtained with the same argument using 
commutation rules between the covariant derivative $\nabla$ and the rough Laplacian $\nabla^*\nabla.$
\end{proof}

We have already seen that Einstein metric have harmonic Riemann tensor, another example of metric with harmonic tensor are locally conformally flat metric with constant scalar curvature.

\subsubsection{Critical metric}
As noticed by G.Tian and J.Viaclovsky \cite{TV_inven} , another large class of Riemannian metric satisfies these regularity estimates:
\begin{defi}We say that a Riemannian metric is critical if its Ricci tensor satisfies an Bochner's type equality :
\begin{equation}\label{eq:Bochner_ricci}
\nabla^*\nabla \ricci_g+\cour(\ricci_g)=0\,\,\, .
\end{equation}
where $\cour$ is a linear action of the Riemann curvature tensor on the space of symmetric $2$ tensor, 
\end{defi}
\begin{prop}\label{prop:regularcritical} A  manifold $(M^n,g)$ endowed with a complete critical metric is 
is $(\Lambda,k)$ regular for a constant $\Lambda$ that only depends on $n$ and $k$ and on the Bochner formula (\ref{eq:Bochner_ricci}).
\end{prop}
\begin{proof} First,
Indeed  using twice the Bianchi identities, we have (see :\cite[formula 3.7]{Bourguignon_inv}) :
$$\nabla^*\nabla \Rm+\cour(\Rm)=d^\nabla(d^\nabla)^*\Rm =-d^\nabla\widetilde{\,\,d^\nabla\ricci_g\,\,}$$
where $\widetilde{\,\,d^\nabla\ricci_g\,\,}(X,Y,Z)=d^\nabla\ricci_g(Y,Z,X)$.
Now we can use  the coupled elliptic system :
\begin{equation}\label{eq:coupled-Bochner}\left\{
\begin{array}{l}
\nabla^*\nabla \Rm+\cour(\Rm)=-d^\nabla\widetilde{\,\,d^\nabla\ricci_g\,\,}\,\, ,\\
\nabla^*\nabla \ricci_g+\cour(\ricci_g)=0\,\,\, .
\end{array}\right.
\end{equation}

By scaling, we assume that on some geodesic ball $B(x,1)$ and for some $\eps\in (0,1),$ 
we have the following uniform bound on the curvature :
$$\sup_{ÊB(x, 1)} |\Rm|\le \eps^2$$
Then the exponential map is a local diffeomorphism form the unit Euclidean balls $\bB(0,1)\subset (T_xM,g_x)$ to $B(x,1)$
$$\exp_x\,:\, \bB(0,1)\rightarrow B(x,1)$$
Then metric $\bg=\exp_x^*g$ is also critical and has its curvature tensor bounded by $\eps^2$.
We will proved the regularity estimate in the ball $\bB(0,1)$ endowed with the metric $\bg=\exp_x^*g$. Hence we work now on the Riemannian manifold $(\bB(0,1), \bg)$ :

Moreover according to  J.Jost and H.Karcher \cite{JK}, M.Anderson \cite[remark : 2.3i)  ]{And_inv} there is a constant $\delta_n$ such that  around each point $p\in \bB(0,1/2)$ 
there is a harmonic chart on the ball of radius 
$\delta_n$ 
$${\bf x}\,:\, \bB(p,\delta_n)\rightarrow \R^n$$ such that the metric ${\bf x}_*\bg$ has uniform $\mcC^{1,\alpha}$ and $W^{2,n}$ estimate.

Looking at the elliptic equation (\ref{eq:Bochner_ricci}) in these coordinates implies that we have a uniform $W^{2,n}$ bound
$$\|\ricci\|_{W^{2,n}(\bB(p,\delta_n/2))}\le C(n) |\Rm\|_{L^n(\bB(p,\delta_n))}\le C(n)\eps^2$$
So that we get an estimate
$$\|\nabla^2\ricci\|_{L^{n}(\bB(p,\delta_n/2))}\le C(n)\eps^2.$$
If we look now at the elliptic equation
$$\nabla^*\nabla \Rm+\cour(\Rm)=-d^\nabla\widetilde{d^\nabla\ricci_g}$$
then we get similarly 
\begin{equation*}
\begin{split}
\|\Rm\|_{W^{2,n}(\bB(p,\delta_n/4))}&\le 
C(n)\left[ \|\Rm\|_{L^n(\bB(p,\delta_n/2))}+ \|\nabla^2\ricci\|_{L^{n}(\bB(p,\delta_n/2))}\right]\\
&\le C(n)\eps^2.
\end{split}
\end{equation*}  In particular we have a uniform estimate on  $\nabla \Rm$ on these balls  $\bB(p,\delta_n/4)$, 
$$\sup_{\bB(0,\frac{2+\delta_n}{4})}|\nabla\Rm |\le C(n)\eps^2.$$

These argument can be bootstrapped because a uniform bound on $\nabla^j\Rm, j=0,\hdots k$ implies uniform
$\mcC^{k+1,\alpha}$ and $W^{k+2,n}$ estimate on the  metric ${\bf x}_*\bg$ and these estimates on the metric imply a 
$W^{k+2,p}$ estimate on the curvature tensor.
\end{proof}
Some example of critical metric :
\begin{enumerate}[i)]
\item A K\"ahler metric with constant scalar curvature is critical. Indeed if $(M,\omega)$ is a K\"ahler manifold with Ricci form $\rho$, The ricci form is closed of type $(1,1)$ and we have
$$d^*\rho=-d^c\scal_g$$
When the scalar curvature is constant, the Bochner formula on $(1,1)$ forms implies that
$$0=(dd^*+d^*d)\rho=\nabla^*\nabla \rho+\cour(\rho).$$
\item Another important example is the case of Bach flat metric in dimension 4.
\end{enumerate}

\subsection{The point selection lemma}
The following proposition can be found in \cite[Appendix H]{Kleiner_Lott}  and is also known as the 1/4-almost maximum lemma (see the $\lambda$-maximum lemma in
\cite[p. 256]{Gromov}.
\begin{prop}\label{PSL}Assume that  $\varphi\,:\, X\rightarrow \R_+$ is a continuous function on a complete locally compact metric space $(X,d)$. If for some $x_0\in X$ and $r>0$ we have
 $$\varphi(x_0)\ge \frac{1}{r^2}$$ 
 then for any $A>0$ there is a point $\overline{x}\in B(x_0,2Ar)$ such that
$$\varphi(\overline{x})\ge \frac{1}{r^2}$$ and
$$\forall z\in B\left(\overline{x}, A\, \varphi(\overline{x})^{-1/2}\right),\ \varphi(z)\le 4 \varphi(\overline{x}).$$
\end{prop}
\proof Starting from $x_0$ we build inductively a sequence $x_0,x_1...$ \\
If $x_l$ is such that
on $B\left(x_0, d(x_0,x_l)+A\,\varphi(x_l)^{-1/2}\right)$ 
$$\varphi\le 4 \varphi(x_l)$$ then we define
$$x_{l+1}=x_l.$$

If it is not the case then we can find $x_{l+1}$ such that 
$$ d(x_0,x_{l+1})\le  d(x_0,x_l)+\frac{A}{\sqrt{\varphi(x_l)}}$$
and $$\varphi(x_{l+1})\ge 4\varphi(x_l).$$
If the points $x_0,x_1,...,x_N$ are distincts then we get for $l\in \{0,...,N\}$ :
$$\varphi(x_l)\ge 4^l\varphi(x_0)$$ and
$$ d(x_0,x_{l})\le \sum_{k=0}^{l-1}\frac{A}{\sqrt{\varphi(x_k)}}\le 2Ar$$
As $\varphi$ is continuous and $B(x_0, 2Ar)$ compact, the sequence must stabilize.
\endproof
\begin{rem} We only need the fact that $\varphi$ is bounded on the ball $B(x,2Ar)$
\end{rem}
\section{Some $\epsilon$-rigidity\& regularity results}
\subsection{$\epsilon$-quadratic decay}
\begin{thm}\label{theo:quadratic}Let $(M,g)$ be a complete Riemannian manifold whose metric is weakly $(\Lambda,1)$ regular ( where $\Lambda\ge 1$). Let  $\epsilon=\frac{1}{6\Lambda}$. If for some fixed point $o\in M$ we have :
$$\forall y\in M\,\,,\,\,\,\, |\Rm(y)|\le \frac{\epsilon^2}{d(o,y)^2}$$
then the metric $g$ is flat : $\Rm=0$.
\end{thm}
\begin{proof} If the curvature does not vanish identically, then our hypothesis implies that we can find a 
point $x\in M$  where the curvature reached its maximum, in particular : 
$$|\Rm(x)|=\frac{1}{r^2}\and\,Ê\sup_{ B(x,r)}|\Rm|\le \frac{1}{r^2} $$
By $(\Lambda,1)$ regularity, we know that 
$$\sup_{ B(x,r/2)} |\nabla\Rm|\le\Lambda \frac{1}{r^3}$$
In particular, for $\delta=1/(2\Lambda)$, we have for $y\inÊB(x,\delta r)$ :
$$|\Rm(y)|\ge |\Rm(x)|-\delta r\Lambda \frac{1}{r^3}\ge \frac{1 }{2}|\Rm(x)|=\frac{1}{2r^2}\,\, .$$
We have supposed  $$|\Rm(x)|\le \frac{\epsilon^2}{d(o,x)^2},$$
hence $$d(o,x)\le \epsilon r\,\,,$$
and when $y\in \partial B(x,\delta r)$, we have 
$d(o,y)\ge d(y,x)-d(o,x)\ge \delta r-\epsilon r$  and
$$\frac{1}{2r^2}\le |\Rm|(y)\le \frac{\epsilon^2}{d(o,y)^2}\le \frac{\epsilon^2}{(\delta-\epsilon)^2 r^2},$$
Our choice of $\delta=3\epsilon$ implies that 
$$\frac{\epsilon^2}{(\delta-\epsilon)^2}=\frac{1}{4},$$
hence the result.
\end{proof}

\subsection{$L^{\frac{n}{2}}$ $\epsilon$-regularity}
\begin{thm} \label{n2reg} Let $(M,g)$ is  a complete Riemannian manifold whose metric is $(\Lambda,1)$ regular for some $\Lambda\ge 1$. There is  a constant $\epsilon(\Lambda,n)>0$ such that if for some $x\in M$ and $r>0$ we have
\begin{enumerate}[i)]
\item $\forall y\in B(x,\frac 34 r), \,\,Ê\forall s\in (0,r/4), \vol B(y,s)\ge vs^n$
\item  $\int_{B(x,r)} |\Rm|^{\frac n2}(y)dy\le \eps(\Lambda,n) v$\end{enumerate}
then
$$\sup_{B(x,\frac12 r)} |\Rm|\le \frac{16}{ r^2} \left(\frac{1}{v\,\,\epsilon(\Lambda,n)}\, \int_{B(x,r)} |\Rm|^{\frac n2}(y)dy\right)^{\frac 2n}.$$
\end{thm}
\begin{proof} Assume that there is a point $z\in B(x,\frac12 r)$ such that 
$$|\Rm|(z)\ge\frac{\mu^2}{r^2}$$ where
$\mu\in (0, 4]$. By the point selection lemma (with $A=\mu/8$), we find a point $y\in B(z,\frac14 r)\subset B(x,\frac34 r)$ such that
$$|\Rm(y)|=\frac{1}{\rho^2}\ge \frac{\mu^2}{r^2}$$
and
$$\sup_{B\left(y,2A(\frac{\rho}{2})\right)}|\Rm|\le \frac{4}{\rho^2}.$$
By  $(\Lambda,1)$ regularity, we get 
$$\sup_{B(y,A\rho/2)}|\nabla\Rm|\le \Lambda \frac{8}{2A\rho^3}= \frac{4\,\Lambda}{A\rho^3}.$$
As in the proof of of the theorem \ref{theo:quadratic}, if we let $$\delta= \frac{A}{8\Lambda}=\frac{\mu}{64\Lambda}$$
then  on the ball $B(y, \delta\rho)\,\,:$
$$|\Rm|\ge \frac{1}{2\rho^2}.$$
Hence we get
\begin{equation*}\begin{split}
\int_{B(x,r)} |\Rm|^{\frac n2}(\sigma)d\sigma&\ge \int_{B(y,\delta\rho)} |\Rm|^{\frac n2}(\sigma)d\sigma\\
&\ge  \frac{\vol B(y, \delta\rho)}{2^{\frac{n}{2}}\rho^n}\\
&\ge v \left(\frac{\delta}{\sqrt{2}}\right)^n\,\,.
\end{split}\end{equation*}
For
$$\epsilon(\Lambda,n)= \left(\frac{1}{16\Lambda\sqrt{2}}\right)^n,$$
we get that when  $\int_{B(x,r)} |\Rm|^{\frac n2}(y)dy\le \eps(\Lambda,n) v$, we can not find
a point $z\in B(x,\frac12 r)$ such that 
$$|\Rm|(z)\ge\frac{16}{r^2}.$$   
Moreover when $z\in B(x,\frac12 r)$ then for 
$\mu^2=r^2|\Rm|(z)$ we get :
$$v \epsilon(\Lambda,n) \left(\frac{\mu}{4}\right)^n=v \left(\frac{\mu}{64\sqrt{2} \Lambda}\right)^n\le\int_{B(x,r)} |\Rm|^{\frac n2}(\sigma)d\sigma.$$
 \end{proof}

\begin{rems}
\begin{enumerate}[i)]
\item For Einstein manifold, this results is due to M. Anderson (\cite{And_JAMS}) : assume that $$\ricci_g=(n-1)\tau g$$ and note by 
$V_\tau(r)$ the volume of a geodesic ball of radius $r$ in the simply connected complete Riemannian $n$-manifold with constant sectional curvature $\tau$, then  the Bishop-Gromov inequality implies that for 
$y\in B(x,\frac 34 r)$ and  $s\in (0,r/4)$ we have :
$$ \vol B(y,s)\ge \frac{V_\tau(s)}{V_\tau(2r)}\, \vol B(y,2r)\ge\frac{V_\tau(s)}{V_\tau(2r)}\, \vol B(x,r)$$
Hence when $|\tau|r^2\le 1$, our proof of  theorem (\ref{n2reg})  shows that  the above hypothesis $i)$ and $ii)$ can be gathered in a single one :
$$\frac{V_\tau(r)}{\vol B(x,r)}\int_{ B(x,r)} |\Rm|^{\frac n2}(\sigma)d\sigma\le \epsilon(n).$$
\item For critical metric and in dimension 4, this result has been also proven G.Tian and J.Viaclovsky (\cite[theorem1.2]{TV_cmh}). In fact, this result was a refinement of a earlier result in (\cite[theorem 3.1]{TV_inven}) where the hypothesis $i)$ was replaced by a Sobolev inequality :
$$\forall \varphi\in C^\infty_0(B(x,r)),\,\,\, \|\varphi\|_{L^{\frac{2n}{n-2}}}\le A\|d\varphi\|_{L^2}.$$

And according to (\cite{Aku} or \cite{Carron_smf}), such a Sobolev inequality implies a lower bound on the volume on geodesic ball :
if $B\subset B(x,r)$ is a geodesic ball of radius $r(B)$ then 
$$\vol B\le C(n)\left(\frac{r(B)}{A}\right)^n.$$
It should also be noticed that the main argument in the proof of the result of G.Tian and J.Viaclovsky was also a point selection lemma that relies a priori to the 
$\epsilon$ regularity result on \cite{TV_inven}, that is the proof relies on a intricate deGeorgi-Moser-Nash iteration scheme argument.
The results of G.Tian and J.Viaclovsky has been extended by X-X. Chen and B.Weber (\ref{ChenWeber}) in two directions :
 for extremal K\"ahler metric and in dimension $n>4$. Now from the proof of  (\cite[proposition 3.1]{TV_cmh}), it is clear that the $\epsilon$-regularity result of 
X-X. Chen and B.Weber (see \cite[theorem 4.6]{ChenWeber}) implies the above $\epsilon$ regularity result. But our proof is shorter and doesn't rely on
deGeorgi-Moser-Nash iteration scheme argument but on quite classical elliptic estimate. 
\item
Eventually, it should be noticed that it is clear that we get estimate on the covariant derivative of the Riemann tensor $\nabla^j\Rm$ , $j=1,\hdots\,, k$,  if we assume that the metric is $(\Lambda,k)$ regular.
\end{enumerate}
\end{rems}
This result also implies some $\epsilon$-$L^{\frac{n}{2}}$ rigidity result :
\begin{cor}Let $(M,g)$ is  a complete Riemannian manifold whose metric is $(\Lambda,1)$ regular for some $\Lambda\ge 1$. Assume that :
\begin{enumerate}[i)]
\item $\forall x\in M$ and $\forall r>0\,,\, \vol B(x,r)\ge vr^n$
\item  $\int_{M} |\Rm|^{\frac n2}(y)dy\le \eps(\Lambda,n) v$\end{enumerate}
Then
$$\Rm=0.$$
\end{cor}

\subsection{$\epsilon$-$L^{p}$ regularity}The above argument can be extended to other $L^p$ estimates on  the curvature :
\begin{thm} \label{preg}Let  $(M,g)$ be   a complete Riemannian manifold whose metric is $(\Lambda,1)$ regular for some $\Lambda\ge 1$. Let $p>0$. 
For any $x\in M$ and $r>0$ we let \footnote{where $B$ runs over all the geodesic ball of radius $r(B)$ included in $B(x,r)$.}:
$$\mcM(x,r):=\sup_{B\subset B(x,r)}\left( \frac{r(B)^{2p}}{\vol B}\int_B |\Rm|^p\right)^{\frac{1}{p}}.$$

There is  a constant $\epsilon(\Lambda,p)>0$ such that if for some $x\in M$ and $r>0$ we have
 $$ \mcM(x)\le \epsilon(\Lambda,p)$$ 
 then
 $$\sup_{B(x,\frac12 r)} |\Rm|\le \frac{16}{\epsilon(\Lambda,p) r^2} \mcM(x,r)$$
\end{thm} And we also get the following $\epsilon$-$L^{p}$ rigidity result :
\begin{cor}Let $(M,g)$ is  a complete Riemannian manifold whose metric is $(\Lambda,1)$ regular for some $\Lambda\ge 1$. Assume that :
$\forall x\in M$ and $\forall r>0$ :
  $$ \frac{r^{2p}}{\vol B(x,r)}\int_{B(x,r)} |\Rm|^{p}(y)dy\le \eps(\Lambda,p)^p $$
Then
$$\Rm=0.$$
\end{cor}
It is also clear that these results together with \cite[theorem 4.1]{TV_inven} gives some conditions that implies finiteness of the number of ends and that each end is ALE of order 0, but we prefer to refrain from stating it.

\section{Almost maximal volume growth}
With the point selection lemma, we are going to give an alternative proof of the following (slightly improved) result of Anderson \cite{And_inv}:

\begin{thm} \label{volcur}There are constant $\epsilon(n)>0$ and $C(n)$ such that if
$(M^n,g)$ is a complete Ricci flat manifold and $x\in M$ and $r>0$ are such that 
\footnote{$\omega_n$ is the volume of the unit Euclidean ball.}
$$\vol B(x,r)\ge \omega_n(1-\epsilon_n)r^n$$ then
$$\sup_{B(x,r/2)}|Rm|\le \frac{C(n)}{r^{2}}\sup_{y\in B(x,\frac34 r)}\left( \frac{\omega_n r^n-\vol B(y,r)}{r^n}\right)^{\frac14}.$$
\end{thm}
This theorem has the following corollary 

\begin{cor}\label{rigvol}If $(M^n,g)$ is a complete Ricci flat manifold
such that 
$$\lim_{r\to\infty} \frac{\vol B(x,r)}{r^n}\ge \omega_n(1-\epsilon_n)$$
then $(M^n,g)$  is isometric to the Euclidean space $\R^n$.
\end{cor}

Anderson has shown first the corollary \ref{rigvol} with an argument by contradiction and then he deduced (also by contradiction) 
an estimate for the $\mcC^{1,\alpha}$-harmonic radius when the volume of the geodesic ball is almost maximal under a uniform bound on the Ricci curvature. When the manifold is Einstein, the elliptic regularity of the Einstein equation implies a bound
on the curvature. For Einstein metric, our curvature estimate is more precise  . Here we are going  to show the theorem \ref{volcur}, the corollary \ref{rigvol} is then straightforward.

\begin{proof}
Again assume that there is a point $z\in B(x,r/2)$ such that such that 
$$|\Rm|(z)\ge \frac{\mu^2}{r^2}$$ where
$\mu\in (0, 4]$. By the point selection lemma (with $A=\mu/8$) we find a point $y\in B(z,\frac14 r)\subset B(x,\frac34 r)$ such that
$$|\Rm(y)|=\frac{1}{\rho^2}\ge \frac{\mu^2}{r^2}$$
and
$$\sup_{B\left(y,2A(\frac{\rho}{2})\right)}|\Rm|\le \frac{4}{\rho^2}.$$
By  $(\Lambda,7)$ regularity, we get  for $j=1,\,\hdots\, ,7.$ :
\begin{equation}
\label{boundnabla}\sup_{B(y,\mu\rho/16)}|\nabla^j\Rm|\le  \frac{C(n)}{(\mu\rho)^j\rho^2}.\end{equation}
According to A. Gray and L. Vanhecke, we know the asymptotic expansion of 
the volume of geodesic balls \cite{GrayI},\cite[Theorem 3.3]{GrayVan} :
$$\vol B(y,r)=\omega_nr^n\left(1-\frac{1}{120(n+2)(n+4)} |Rm(y)|^2r^4+O(r^6)\right)$$
We are going to estimate the "$O(r^6)$" . The first step is to remark that if $\bB(s)$ is the Euclidean ball of radius
$s$ in $(T_yM,g_{y})$ then
$$\exp_{y}\,:\, \bB(A\rho)\rightarrow B(y,A\rho)$$ is an immersion, hence for $\bar g=\exp_{y}^*g$,
we get for all $r\le A \rho$: 
$$\vol B(y,r)\le\vol_{\bar g} \bB(\rho).$$
The estimation (\ref{boundnabla}) and the Jacobi equation implies that if
$t\mapsto J(t)$ is a Jacobi field along the geodesic 
$t\mapsto \exp_{y}(tv)$
with $|v|=1$, $J(0)=0$ and 
$|J'(0)|=1$ then for all
$t\in [0,\mu\rho/16]$ and $l\in \{0,...,7\}$\,\,:
$$\left|\frac{d^l}{dt^l} J(t)\right| \le B_n|Rm(y)|(\mu\rho) ^{3-l}$$
Then Gray\&Vanhecke's computation leads to 
$$\forall s\in (0,\mu\rho/16)\,\,,\,\, \vol_{\bar g} \bB(s)=\omega_ns^n\left(1-\frac{1}{120(n+2)(n+4)} |Rm(y)|^2s^4+\delta(s)\right)$$
where for some constant $D_n>1$ depending only on the dimension $n$ :
$$|\delta(s)|\le D_n s^6|Rm(y)|(\mu\rho)^{-4}$$
We choose
$s=\eta_n \mu^2 \rho$ such that
$$D_n s^6|Rm(y)| (\mu\rho)^{-4}=\frac{1}{240(n+2)(n+4)} |Rm(y)|^2s^4$$
i.e.
$$\eta_n^2=\frac{1}{240(n+2)(n+4)D_n}$$
Then we get for $\sigma=\eta_n \mu^2 \rho$
$$ \frac{\vol B(y,r)}{r^n}\le \frac{\vol B(y,\sigma)}{\sigma^n}\le \frac{\vol_{\bar g} \bB(\sigma)}{\sigma^n}\le \omega_n\left(1-\frac{\eta_n^4\mu^8}{240(n+2)(n+4)}\right).$$
\end{proof}
\subsection{A sphere theorem}
With the same idea, we can give a direct proof of the following result
\begin{thm}\label{vol_rig}
There is a $\eps_n>0$ such that if $(M^n,g)$ is closed Einstein manifold with positive scalar curvature:
$$\ricci_g=(n-1)g$$
 and 
 $$\frac{\vol(M,g)}{\vol \bS^n}\ge 1-\eps_n$$
 then $(M,g)$ is isometric to the round sphere $\bS^n$.
\end{thm}

Perhaps, there is a nice optimal volume pinching for Einstein metric with positive scalar curvature, a nice result in this direction has been proved
by M. Gursky (\cite{Gursky}) any non standard Einstein metric $g$ on the sphere $\bS^4$ must satisfies
$$\frac{\vol(\bS^4,g)}{\vol \bS^4}\le \frac13\,\,\,.$$
The same proof will also prove a local version of this result : for $r\in [0,\pi]$, we denote by $V_1(r)$ the volume of a geodesic ball in $\bS^n$ :
$$V_1(r)=\vol (\bS^{n-1})\int_0^r (\sin(t))^{n-1}dt.$$
\begin{thm}
There is a $\eps_n>0$ such that if $(M^n,g)$ is closed Einstein manifold with positive scalar curvature:
$$\ricci_g=(n-1)g$$
 and such that for some $r\in (0,\pi]$ and all $x\in M$ :
 $$\frac{\vol(B(x,r))}{V_1(r)}\ge 1-\eps_nr^4$$
 then $(M,g)$ has constant sectional curvature.
\end{thm}

These theorems are consequence of a result of M. Anderson and of the isolation of the round metric amongst Einstein metric.
Indeed, a consequence of Anderson's result (\cite[theorem 1.2]{And_inv}) is the following :

  For $\delta>0$, we can choose $\epsilon(n,\delta)>0$ such that 
the hypothesis $$\ricci_g=(n-1)g\,\,\,  \mathrm{and }\,\,\, \frac{\vol(M,g)}{\vol \bS^n}\ge 1-\eps(n,\delta)$$
implies that  the sectional curvature  of $g$ are in a interval $(1-\delta, 1+\delta)$. Now according to \cite{huisken}, \cite{margerin}, \cite{BW},\cite{Brendle},
 we know that a Einstein metric with  sectional curvature in the interval $(\frac12, 2)$ has constant sectional curvature.
 If we don't care about the optimal value of the pinching condition such a rigidity result can be easily proven with the maximum principle.
 
 Indeed the Weyl tensor $\We$ of an Einstein metric satisfies a Bochner formula (\cite[Proposition 4.2]{Bourguignon_inv},\cite{Singer}):
 $$\nabla^*\nabla \We+\frac{2\scal_g}{n}\We=\We*\We$$
 where $\We*\We$ is a  quadratic expression in the Weyl tensor.
 Hence if $\ricci_g=(n-1)g$, we obtain that the length of the Weyl tensor satisfies :
 $$\Delta|\We|^2+4(n-1)|\We|^2=2\langle\We, \We*\We\rangle-2|\nabla \We|^2$$ Hence at a point where the  length of the Weyl tensor reaches its maximum, we have :
 $$4(n-1)|\We|^2\le \Delta|\We|^2+4(n-1)|\We|^2=2\langle\We, \We*\We\rangle\le c(n) |\We|^3.$$
 Hence either $\We=0$ or $\max_{x\in M} |\We(x)|\ge \frac{2(n-1)}{c(n)}.$
\begin{proof}
 We use again the same idea to proved the above theorems. Assume that $(M^n,g)$ is a closed Einstein manifold with positive scalar curvature:
$$\ricci_g=(n-1)g$$
and that the sectional curvature of $g$ are not constant, then we know that 
$$\max_M |\We| \ge  \frac{2(n-1)}{c(n)}.$$
Let $x\in M$ be a point where this maximum is reached :
$$\frac{1}{\rho^2}=|\We(x)|=\max_M |\We|.$$
By regularity, we obtain estimates on all the covariant derivative of the Weyl tensor : for $j\in \{1,\,\hdots,\, 7\}$
$$\max_M |\nabla^j\We|\le C(n)\frac{1}{\rho^j}$$
(Recall that the diameter of $M$ is bounded by $\pi$ and that $\rho^2\le  \frac{2(n-1)}{c(n)}.$)

The same argument using the computations of Gray and Vanhecke show that for some constant $\delta_n>0$
and for all $s\in (0,\delta_n \rho)$ :
$$\frac{\vol(B(x,s))}{V_1(s)}\le 1-\frac{1}{240(n+2)(n+4)}\left(\frac{s}{\rho}\right)^4.$$

Then the Bishop-Gromov comparison principle implies then that :
$$\frac{\vol(M,g)}{\vol \bS^n}=\frac{\vol(B(x,\pi))}{V_1(\pi)}\le \frac{\vol(B(x,\delta_n \rho))}{V_1(\delta_n \rho)}\le 1-\frac{\delta_n^4}{240(n+2)(n+4)}.$$
 It also implies that for $r\in (\delta_n \rho, \pi]$
 $$\frac{\vol(B(x,r))}{V_1(r)}\le 1-\frac{\delta_n^4}{240(n+2)(n+4)}\le 1-\frac{\delta_n^4\, r^4}{240(n+2)(n+4)\pi^4}.$$
 and because $\rho^2\le  \frac{2(n-1)}{c(n)}$, we have a constant $\eta_n$ such that for all 
 $r\in (0, \pi]$ :
 $$\frac{\vol(B(x,r))}{V_1(r)}\le 1-\eta_n r^4.$$

\end{proof} 
 \subsection{Another rigidity result}
 The same argument can be used to prove a volume rigidity result when the scalar curvature is zero and when the second term in the asymptotic expansion in the volume of geodesic balls has a definite sign\,: 
 \begin{thm}There is a constant $\eps_n>0$, such that when $(M^n,g)$ is a complete locally conformally flat manifold with zero scalar curvature of dimension $n\ge 4$ 
 such that for some $v>0$ :
 $$\forall x\in M, \,\, \forall r>0\,\,\,\,\, :\,\,\, vr^n\le \vol B(x,r)\le \omega_nr^n(1+\eps_n v^4)$$
 then $(M^n,g)$ is isometric to the Euclidean space $\R^n$.
\end{thm}
\begin{proof}
Indeed at a $4$ almost maximal point of the length of the Riemann curvature tensor, we have
$$ |\Rm(x)|=\frac{1}{\rho^2}\,\,\,  \mathrm{and} \,\,\, \max_{B(x,\rho)} |\Rm|\le \frac{4}{\rho^2}.$$
because $\vol B(x,r)\ge vr^n$, then Cheeger's estimate of the injectivity radius (\cite{Cheeger},\cite[theorem 4.2]{CGT})
implies that the injectivity radius at $x$ is bounded from below :
$$\mathrm{inj}_x\ge \eta_n v \rho.$$
Again if we denote by $\bB(s)$  the Euclidean ball of radius
$s$ in $(T_xM,g_{x})$ then
$$\exp_{x}\,:\, \bB(\eta_n v \rho)\rightarrow B(y,\eta_n v \rho)$$ is a diffeomorphism (Note that
our hypothesis implies in particular that 
$v\le \omega_n$). Hence  for $\mathbf{g}=\exp_{c}^*g$,
we get for all $r\le\eta_n v  \rho$: 
$$\vol B(y,r)=\vol_{\mathbf{ g}} \bB(r).$$

When the metric is locally conformally flat with zero scalar curvature, Gray\&Vanhecke's computation is that 
$$\vol B(x,r)=\omega_n r^n\left(1+\frac{2n-7}{90(n^2-4)(n+4)}|\ricci_g(x)|^2r^4+O(r^6)\right)$$
The same arguments implies that for some $\delta_n>0$ and $\eps_n>0$ , we have for $s=\delta_n v \rho$
$$\vol B(x,s)\ge \omega_n s^n\left(1+v^4\eps_n\right).$$
\end{proof}
\begin{rem} Using the \cite[Corollary 3.4]{GrayVan} in dimension $3$,
the same proof furnishes that there is a $\eps(\Lambda)>0$ such that if $(M,g)$ is complete $(\Lambda,7)$ regular $3$-manifold with zero scalar curvature such that 
$$\forall x\in M,\,\, \forall r\ge 0 \,\,:\,\, \vol B(x,r)\ge \omega_nr^3(1-\eps),$$ then $(M,g)$ is isometric to the Euclidean space $\R^3$.
\end{rem}

\end{document}